\documentclass[10pt]{amsart}

\usepackage{amsmath}
\usepackage{mathrsfs}
\usepackage{amssymb}
\usepackage{amsfonts}
\usepackage{amsthm}
\usepackage{enumerate}
\usepackage[top=1.5in, bottom=1.5in, left=1.25in, right=1.25in]{geometry}
\usepackage{mathtools}
\usepackage{esint}
\usepackage[hidelinks]{hyperref}
\usepackage{xcolor}
\usepackage{comment}

\usepackage[style=trad-plain, hyperref=true]{biblatex}
\providecommand{\N}{\mathbb{N}}
\providecommand{\Z}{\mathbb{Z}}

\providecommand{\R}{\mathbb{R}}
\providecommand{\C}{\mathbb{C}}

\providecommand{\E}{\mathbb{E}}

\providecommand{\T}{\mathbb{T}}
\providecommand{\leqsim}{\lesssim}
\providecommand{\geqsim}{\gtrsim}

\renewcommand{\H}{\mathcal{H}}
\renewcommand{\:}{\colon}

\numberwithin{equation}{section}

\newtheorem{theorem}{Theorem}[section]
\newtheorem{proposition}[theorem]{Proposition}

\newtheorem{lemma}[theorem]{Lemma}

\theoremstyle{definition}

\theoremstyle{remark}

\title{Sárközy's Theorem for Fractional Monomials}
\author{Maximilian O'Keeffe}
\date{}

\begin{document}
\maketitle

\begin{abstract}
Suppose $A$ is a subset of $\{1, \dotsc, N\}$ which does not contain any configurations of the form $x,x+\lfloor n^c \rfloor$ where $n \neq 0$ and $1<c<\frac{6}{5}$. We show that the density of $A$ relative to the first $N$ integers is $O_c(N^{1-\frac{6}{5c}})$. More generally, given a smooth and regular real valued function $h$ with ``growth rate" $c \in (1,\frac{6}{5})$, we show that if $A$ lacks configurations of the form $x,x \pm \lfloor h(n) \rfloor$ then $\frac{|A|}{N} \ll_{h,\varepsilon} N^{1-\frac{6}{5c}+\varepsilon}$ for any $\varepsilon>0$.
\end{abstract}

\section{Introduction}

\subsection{Background}

Fix a subset $S$ of $\N$ and let $\delta_S(N)$ denote the density $\frac{|A|}{N}$ of the largest set $A \subseteq \{1,\dotsc,N\}$ which contains no configurations of the form $\{x,x+s\}$ where $s \in S$. Equivalently, $\delta_S(N)$ is the density of the largest set $A$ whose difference set \[ A-A = \{a-a' : a,a' \in A\} \] does not intersect $S$. It is natural to ask, for various sets $S$, what the behaviour of $\delta_S(N)$ is as $N \to \infty$.

Taking $S_2$ to be the set of squares, S\'ark\"ozy \cite{sarkozy1978difference1} and Furstenberg \cite{furstenberg1977ergodic} began this line of investigation and independently proved that $\delta_{S_2}(N) \to 0$ as $N \to \infty$. While Furstenberg's proof was qualitative in nature and used ergodic theory, S\'ark\"ozy proved a quantitative bound \[ \delta_{S_2}(N) \ll \frac{(\log \log N)^\frac{2}{3}}{(\log N)^\frac{1}{3}} \] via Fourier analysis and the circle method. S\'ark\"ozy's bound was improved by Pintz, Steiger, and Szemer\'edi in 1988 \cite{pintz1988sets} to \[ \delta_{S_2}(N) \ll (\log N)^{-c \log \log \log \log N} \] for some absolute constant $c>0$ and more recently, by Bloom and Mayndard \cite{bloom2022new} to \[ \delta_{S_2}(N) \ll (\log N)^{-c \log \log \log N} \] for some absolute constant $c>0$.

Various extensions of these results have been established. For example, taking $S_k$ to be the set of $k$\textsuperscript{th} powers, Balog, Pelik\'an, Pintz, and Szemer\'edi \cite{balog1994difference} showed that \[ \delta_{S_k}(N) \ll_k (\log N)^{-c\log \log \log \log N} \] for an absolute constant $c>0$. Moreover, if $P \in \Z[n]$ is an intersective polynomial\footnote{We say that $P$ is intersective if for every $q \in \N$ there exists $n \in \Z$ such that $P(n)$ is divisible by $q$. This condition is necessary as otherwise the difference set of $q\Z$ does not intersect $S_P$.} and \[ S_P = \{ P(n) : n \in \N \} \] then Lucier \cite{lucier2006intersective} showed that \[ \delta_{S_P}(N) \ll_P \left(\frac{(\log \log N)^c}{\log N} \right)^{(\deg P-1)^{-1}} \] for some absolute constant $c \geq 2$, with Lyall and Magyar \cite{lyall2009polynomial} proving that one can take $c=1$. Moreover, Arala \cite{arala2024maximal} recently proved that \[ \delta_{S_P}(N) \ll_P (\log N)^{-c \log \log \log N} \] for some $c>0$.

Besides polynomials, in 1978 S\'ark\"ozy \cite{sarkozy1978difference3} also considered the case where \[ S = \{p-1 : p \text{ is prime}\} \] is the set of shifted primes and proved that \[ \delta_S(N) \ll \frac{(\log \log \log N)^3 \log \log \log \log N}{(\log \log N)^2}. \] This bound was improved by Lucier \cite{lucier2008difference} to \[ \delta_S(N) \ll \left(\frac{(\log \log \log N)^3 \log \log \log \log N}{\log \log N} \right)^{\log \log \log \log \log N}, \] by Ruzsa and Sanders \cite{ruzsa2008difference} to \[ \delta_S(N) \ll e^{-c(\log N)^\frac{1}{4}}, \] and by Wang \cite{wang2020theorem} to \[ \delta_S(N) \ll e^{-c (\log N)^\frac{1}{3}}. \]

A common feature of the results outlined so far is that they all rely on some form of an increment strategy. Roughly speaking, the idea is to show that if $A$ has density $\delta$ in $\{1,\dotsc,N\}$ and $A-A$ does not intersect $S$, then there is a ``large" arithmetic progression $P$ such that $\frac{|A \cap P|}{|P|}$, the relative density of $A$ in $P$, is bounded below by $\delta+c \delta^{O(1)}$ for a suitable $c>0$. On iterating this argument, a contradiction is obtained as the relative density eventually exceeds 1. This iteration does not lend itself to obtaining bounds for $\delta_S(N)$ of the form $N^{-\delta}$ for a fixed $\delta>0$. However, Green \cite{green2022s} was recently able to obtain such a bound in the case where $S$ is the set of shifted primes by using an argument which does not rely on any density increment.

\subsection{Intersective and Van der Corput Sets}

The key difference in Green's paper, which allows for the polynomial saving in $N$, is to shift perspective from studying the intersectivity property of these sets to studying the so-called \emph{Van der Corput} property. We give a brief outline of these properties, but for a more in depth discussion see \cite{montgomery1994ten}.

We say that $S$ is intersective if $S \cap (A-A)$ is non-empty for any set $A$ of positive upper density, i.e. \[ \Bar{d}(A)= \limsup_{N \to \infty} \frac{|A \cap \{1,\dotsc,N\}|}{N}>0. \] All the preceding results show that the relevant sets considered are intersective by quantifying the rate of convergence of $\frac{|A \cap \{1,\dotsc,N\}|}{N}$ to zero whenever $S \cap (A-A)$ is empty, by bounding $\delta_S(N)$. A set $S$ is called a Van der Corput set if a sequence $(u_n)_{n \in \N}$ is uniformly distributed modulo 1 whenever the sequence $(u_{n+h}-u_n)_{n \in \N}$ is uniformly distributed modulo 1 for all $h \in S$. Restating a theorem of Van der Corput, the set $\N$ is a Van der Corput set. The proof of this fact can be deduced by Weyl's criterion for uniform distribution coupled with Van der Corput's inequality. Van der Corput's inequality has the following generalisation which can be found, for example, in \cite{montgomery1994ten}.

\begin{lemma} \label{lem:generalisedvdc}
Let \[ T(\xi) = a_0 + \sum_{h=1}^H a_h \cos(2\pi h\xi) \] be a non-negative real trigonometric polynomial satisfying $T(0)=1$. Then for any $y_1,\dotsc,y_N \in \C$ we have \begin{equation} \label{eqn:generalisedvdc} \left| \sum_{n=1}^N y_n \right|^2 \leq (N+H) \left(a_0 \sum_{n=1}^N |y_n|^2 + \sum_{h=1}^H |a_h| \left| \sum_{n=1}^{N-h} y_{n+h} \overline{y_n} \right| \right). \end{equation}
\end{lemma}

Note that by taking $a_h = \frac{1}{H+1}(1-\frac{h}{H+1})$ for all $h$ we recover Van der Corput's inequality.

Using this lemma and Weyl's criterion, if $(u_{n+h}-u_n)_{n \in \N}$ is uniformly distributed for all $h \in S$, and there exists a trigonometric polynomial of the kind appearing in Lemma \ref{lem:generalisedvdc} whose coefficients $a_n$ are supported on $S$, then we have \[ \limsup_{N \to \infty} \left| \frac{1}{N} \sum_{n=1}^N e^{2 \pi i k u_n} \right| \leq a_0^\frac{1}{2} \] for any $k \neq 0$. We see that $u_n$ is distributed, and hence $S$ is a Van der Corput set, if \[ \ \inf_{T \in \mathcal{T}_S} a_0 =0, \] where $\mathcal{T}_S$ denotes the set of trigonometric polynomials $T$ of the kind in Lemma \ref{lem:generalisedvdc} with coefficients supported on $S$. The connection between Van der Corput sets and intersective sets is the relation \[ \sup_{\substack{A \subseteq \N \\ S \cap (A-A) = \emptyset}} \Bar{d}(A) \leq \inf_{T \in \mathcal{T}_S} a_0, \] which follows by taking \[ y_n = \Bar{d}(A)^{-\frac{1}{2}} 1_A(n) \] in (\ref{eqn:generalisedvdc}). Thus every Van der Corput set is an intersective set. To quantify the notion of a Van der Corput set in the same way that $\delta_S(N)$ quantified the notion of being an intersective set, we define \[ \gamma_S(N) = \inf_{T \in \mathcal{T}_{S \cap \{1,\dotsc,N\}}} a_0. \] Then by (\ref{eqn:generalisedvdc}) again it follows that $\delta_S(N) \leq 2 \gamma_S(N)$, so it now suffices to bound $\gamma_S(N)$. Green showed that $\gamma_S(N) \ll N^{-\delta}$ in the case that $S$ is the shifted primes.

\subsection{Statement of Results}

In this paper we will consider sets of the form \[ S_h = \{ \pm \lfloor h(n) \rfloor : n \in \N \} \] where $h$ is a smooth and slowly varying function which, roughly, grows like $t^c$ with $1<c<\frac{6}{5}$. As a special case, we can take $h(t)=t^c$, in which case we obtain bounds for sets whose difference set does not intersect with Piatetski--Shapiro sets.

In 1997 S\'ark\"ozy and Rivat \cite{rivat1997sequence} were able to obtain a power savings for such sets when $h(t)=t^c$ with $c>1$ non-integer.

\begin{theorem}[S\'ark\"ozy--Rivat] \label{thm:SR}
Let $c>1$ be non-integer. There exists $N_c \in \N$ such that for all $N>N_c$ and any $A,B \subseteq \{-N,\dotsc,N\}$ satisfying \[ (|A||B|)^\frac{1}{2} \geqsim_c N^{1-(6c^3(\log c+14))^{-1}} \] the sumset $A+B$ intersects $S_c$.
\end{theorem}

See Theorem 3 of \cite{rivat1997sequence} for this result. Although the authors only state it for non-integer $c>12$, they prove it for all non-integer $c>1$ (see the remarks after Lemma 5 of \cite{rivat1997sequence}). Theorem \ref{thm:SR} is also only stated for subsets of $\{1,\dotsc,N\}$ but the proof does not rely on this fact. Thus one can take $A \subseteq \{1,\dotsc,N\}$ and $B=-A$ to show that if $A \subseteq \{1,\dotsc,N\}$ lacks configurations of the form $\{x,x+\lfloor n^c\rfloor \}$ with $n \neq 0$ then one has the power savings bound \begin{equation} \label{eqn:SR} |A| \leqsim_c N^{1-(6c^3(\log c+14))^{-1}}. \end{equation}

Our results concern power savings bounds for sets whose difference intersects $S_c$ when $c$ is sufficiently close to $1$.

\begin{theorem} \label{thm:fracmonom}
Let $1<c<\frac{6}{5}$ and let \[ S_c = \{ \lfloor n^c \rfloor : n \in \N\}. \] Then \[ \gamma_{S_c}(N) \ll_{c,\varepsilon} N^{-(\frac{6}{5c}-1)+\varepsilon}. \] Thus if $A \subseteq \{1, \dotsc, N\}$ lacks configurations of the form $\{x,x+\lfloor n^c \rfloor\}$ with $n \neq 0$ then \begin{equation} \label{eqn:fracmonom} |A| \ll_{c} N^{1-(\frac{6}{5c}-1)}. \end{equation}
\end{theorem}

We emphasise that Theorem \ref{thm:fracmonom} only provides a power savings bound for $1<c<\frac{6}{5}$, whereas Theorem \ref{thm:SR} is valid for any non-integer $c>1$. However, if $c<1.1917$ (say) then the bound (\ref{eqn:fracmonom}) improves on (\ref{eqn:SR}) with the improvement getting better the closer $c$ gets to 1.

More generally, suppose $h$ is a \emph{Hardy field function}. See Subsection \ref{subsec:Hardy} for a precise definition but for now we think of any smooth function such that it, and all its derivatives, are eventually monotonic. Suppose its ``growth rate" is equal to $c \in (1,\frac{6}{5})$ (again, see Subsection \ref{subsec:Hardy} for a precise definition of growth rate). Then we will in fact prove the following more general theorem.

\begin{theorem} \label{thm:mainthm}
Let $h$ be a Hardy field function with growth rate $1<c<\frac{6}{5}$ and let \[ S_h = \{ \pm \lfloor h(n) \rfloor : n \in \N \}. \] Then \[ \gamma_{S_h}(N) \ll_{h,\varepsilon} N^{-(\frac{6}{5c}-1)+\varepsilon}. \] Thus if $A \subseteq \{1,\dotsc,N\}$ lacks configurations of the form $\{x, x \pm \lfloor h(n) \rfloor \}$ with $h(n) \neq 0$ then \[ |A| \ll_{h,\varepsilon} N^{1-(\frac{6}{5c}-1)+\varepsilon}. \]
\end{theorem}

Note that Theorem \ref{thm:fracmonom} is almost an immediate corollary of Theorem \ref{thm:mainthm} simply by taking $h(t)=t^c$ with $1<c<\frac{6}{5}$. The introduction of an $\varepsilon$ loss comes from the fact that $h$ may satisfy \[ t^{c-\varepsilon} \ll_\varepsilon h(t) \ll_\varepsilon t^{c+\varepsilon}, \] where one cannot remove the dependence on $\varepsilon$. If one can remove the dependence on $\varepsilon$, as is the case with $h(t)=t^c$, then in our arguments the introduction of $\varepsilon$ is not needed and can thus be removed from Theorem \ref{thm:mainthm}.

Theorem \ref{thm:mainthm} applies to (constant multiples of) functions of the form $h(t)=t^c (\log t)^A$ with $A \in \R$, $h(t) = t^c e^{A (\log t)^B}$ with $A \in \R$ and $B \in (0,1)$, and $h(t) = t^c l_m(t)$ where $l_m$ is the composition of the logarithm with itself $m$ times.

\subsection{Acknowledgments}

The author would like to thank Sean Prendiville for suggesting this problem, Nikos Frantzikinakis for bringing the author's attention to a related paper of S\'ark\"ozy and Rivat \cite{rivat1997sequence} and for helpful discussions, and M\'at\'e Wierdl for helpful comments and corrections.

\section{Preliminaries}

\subsection{Notation}

For a subset $A$ of the integers, we will denote its cardinality by $|A|$. For convenience we will denote $\{1,\dotsc,N\}$ by $[N]$ and we write \[ \E_{n \in [N]} f(n) = \frac{1}{N} \sum_{n=1}^N f(n) \] for any function $f\: \Z \to \C$. We also define $\Delta f(x) = f(x+1)-f(x)$ to be the discrete derivative of $f\:\Z \to \C$. We will use the convenient abbreviation $e(\xi) = e^{2 \pi i \xi}$ for any $\xi \in \T$, where $\T = \R/\Z$ denotes the unit circle with normalised Lebesgue measure. Given two functions $f,g\:\T \to \C$ we define their convolution as \[ (f*g)(\xi) = \int_\T f(\zeta) g(\xi-\zeta) \,d\zeta. \]

We define the floor function of a real number $x$ as \[ \lfloor x \rfloor = \max \{k\in \Z : k \leq x\}. \] Given $N \in \N$, we define \[ F_N^\Z(n) = \left(1-\frac{|n|}{N} \right)_+ = \max \left( 1-\frac{|n|}{N}, 0 \right) \] and \[ F_N^\T(\xi) = \sum_{n=-N}^N F_N^\Z(n) e(n\xi) \] to be the Fej\'er kernels on $\Z$ and $\T$ respectively. We also define the Dirichlet kernel \[ D_N(\xi) = \sum_{n=-N}^N e(n\xi). \] 

For quantities $X$ and $Y$ we will write $X=O(Y)$ or $X \ll Y$ to mean that there is a constant $C>0$ such that $|X| \leq CY$. If the constant depends on any parameters this will be denoted by a subscript in the asymptotic notation.

\subsection{Hardy Field Functions} \label{subsec:Hardy}

In this subsection we provide a definition of a Hardy field and discuss their properties.

Consider the set of smooth functions $h\:(u,\infty) \to \R$, defined on some interval $(u,\infty)$, where $u \in \R \cup \{-\infty\}$. Call two such functions $h_1$ and $h_2$ equivalent if there exists $T>0$ such that $h_1(t)=h_2(t)$ for all $t \geq T$. This is an equivalence relation, and a \emph{germ at infinity} is simply an equivalence class under this notion of equivalence. The set of germs at infinity becomes a ring under the obvious operations $[h_1]+[h_2]=[h_1+h_2]$ and $[h_1][h_2]=[h_1h_2]$. By an abuse of notation we will denote $[h]$ by $h$. A subfield of the ring of germs at infinity which is closed under differentiation is called a \emph{Hardy field}. Common examples of Hardy fields include the field of constant functions on $\R$, the set $\R(x)$ of rational functions, and the so-called \emph{logarithmico-exponential} functions, which is the field generated (under the usual ring operations, differentiation, and function composition) by polynomials, the exponential function, and the logarithm.

Hardy field functions satisfy many nice properties. If $\H$ is a Hardy field and $h \in \H$ then $h$ is eventually positive, negative, or identically zero. This follows from the fact that $h$ must have a multiplicative inverse in $\H$ (as long as $h \neq 0$) and the Intermediate Value Theorem. As a consequence, $\H$ contains no periodic functions, and hence $\lim_{t \to \infty} h(t)$ always exists in $\R \cup \{\pm \infty\}$ for any $h \in \H$. Since derivatives of Hardy field functions are also Hardy field functions, it also follows that $h$ is eventually strictly increasing, constant, or strictly decreasing.

One can define the \emph{growth rate} of a function $h \in \H$ to be \[ \inf \{ \alpha \in \R : |h(t)| \leq t^\alpha \text{ eventually} \}. \] In this paper we will only be concerned with Hardy field functions $h$ whose growth rate $c$ lies in $(1,\frac{6}{5})$. In this case note that for all $\varepsilon>0$ we have $t^{c-\varepsilon} \leq |h(t)| \leq t^{c+\varepsilon}$ for all $t$ sufficiently large (depending on $h$ and $\varepsilon$). Moreover, the growth rate of a Hardy field function behaves well with respect to differentiation and inverses. Indeed, in this case $h^{(j)}$ has growth rate $c-j$ for each $j \in \{1,2\}$ (this can be seen by the Fundamental Theorem of Calculus and the uniqueness of growth rates). Since $h$ has growth rate bigger than 1, it is eventually strictly monotonic and thus has an inverse. The growth rate of $h^{-1}$ is equal to $\frac{1}{c}$.

\subsection{Lemmata}

We will use the following two results in our paper so we record them here for convenience.

The first is a bound on exponential sums with smooth phases.

\begin{lemma}[Van der Corput Second Derivative Test] \label{lem:vdc}
Let $I \subseteq \R$ be an interval and suppose $f\: I \to \R$ is a twice continuously differentiable function satisfying \begin{equation} \label{eqn:vdcnonneg}
    0<\lambda \leq |f''(t)| \leq \alpha \lambda
\end{equation} for all $t \in I$. Then \[ \sum_{n \in I \cap \Z} e(f(n)) \ll \alpha |I| \lambda^\frac{1}{2} + \lambda^{-\frac{1}{2}}. \]
\end{lemma}

Since the phases of the exponential sums considered in this paper will be Hardy field functions with positive growth rates, (\ref{eqn:vdcnonneg}) will be satisfied for any interval contained in $(t_0,\infty)$ with $t_0$ sufficiently large.

The next lemma contains the important non-negativity property that Fej\'er kernels enjoy, and provides a relation between the Fej\'er kernel on $\T$ and Dirichlet kernels.

\begin{lemma} \label{lem:nonnegative}
Let $N \geq 1$. Then we have the following: \begin{enumerate}[(i)]
    \item $F_N^\T(\xi) \geq 0$ for all $\xi \in \T$; \item $F_N^\T = \frac{1}{N} \sum_{n=0}^{N-1} D_n$.
\end{enumerate}
\end{lemma}

\begin{proof}
We begin with (ii). This is a standard calculation which can be obtained by expanding the definition of $D_n$ and changing the order of summation. One can then use (ii) and the geometric sum formula to show that $F_N^\T(\xi) = \frac{1}{N} (\frac{\sin(\pi N\xi)}{\sin(\pi \xi)})^2$, from which (i) is immediate.
\end{proof}

\section{Recasting the Van der Corput Property}

From now on we fix a regular Hardy field function $h$ with growth rate $c \in (1,\frac{6}{5})$. We allow all implied constants to depend on $h$ and $c$ and will suppress these parameters from subscripts. Without loss of generality we can assume that $h$ is eventually positive (and therefore increasing since it has a positive growth rate). Let $N_h \in \N$ such that $h$ is positive and increasing on $[N_h,\infty)$ and on this set denote the inverse of $h$ by $\eta$. We fix the set \[ S_h = \{ \lfloor h(n) \rfloor : n \in \N \}. \]

Following Green \cite{green2022s} we will shift our focus from bounding $\gamma_{S_h}(N)$ directly to finding a function $\Psi_N\:\C \to \R$ which will be used to define a non-negative trigonometric polynomial.

\begin{lemma} \label{lem:findfunction}
Let $N \in \N$. Suppose there is a function $\Psi_N \:\N \to \R$ supported on $S$ and $\delta_1,\delta_2>0$ such that \begin{equation}
    \E_{n \in [N]} \Psi_N(n) \cos(2\pi \xi n) \geq -\delta_1
\end{equation} for all $\xi \in \T$ and \begin{equation}
    \E_{n \in [N]} \Psi_N(n) \geq \delta_2.
\end{equation} Then $\gamma_S(N) \leq \frac{\delta_1}{\delta_1+\delta_2}$. Note we allow $\delta_1$ and $\delta_2$ to depend on $N$.
\end{lemma}

\begin{proof}
Define the trigonometric polynomial \[ T(\xi) = \frac{\delta_1 + \E_{n \in [N]} \Psi_N(n) \cos(2\pi n\xi)}{\delta_1 + \E_{n \in [N]} \Psi_N(n)}. \] Then clearly $T(0)=1$ and \[ \delta_1+\E_{n \in [N]} \Psi_N(n) \geq \delta_1+\delta_2>0 \] so $T(\xi) \geq 0$ for all $\xi \in \T$. Moreover, \[ \frac{\delta_1}{\delta_1+\E_{n \in [N]} \Psi_N(n)} \leq \frac{\delta_1}{\delta_1+\delta_2} \] so it follows by definition that $\gamma_S(N) \leq \frac{\delta_1}{\delta_1+\delta_2}$.
\end{proof}

In Green's paper, both defining the function $\Psi$ and showing that it satisfies the desired properties is very involved. The idea is to take $\Psi$ to be a weighted indicator set of the primes shifted by $\pm 1$ multiplied by a function with a non-negative Fourier transform, such as $F_N^\Z$. Doing this as written would require an understanding of the twin prime conjecture which is one of the reasons why the construction of $\Psi$ is so much more involved.

However, for our purposes we are able to essentially use the ideal construction. We will be able to take $\Psi$ to be the product of the indicator set of $S_h$ with a Fej\'er kernel. Thus with Lemma \ref{lem:findfunction} in hand, Theorem \ref{thm:mainthm} now follows from the following proposition.

\begin{proposition}
Let $N \in \N$ and let \[ \Psi_N(n) = F_N^\Z(n) 1_{S_h}(n). \] Then \begin{equation} \label{eqn:negative}
    \E_{n \in [N]} \Psi_N(n) \cos(2\pi \xi n) \gg_\varepsilon -N^{-\frac{1}{5c}+\varepsilon}
\end{equation} for any $\varepsilon>0$ and \begin{equation} \label{eqn:large}
    \E_{n \in [N]} \Psi_N(n) \gg_\varepsilon N^{-(1-\frac{1}{c}+\varepsilon)}.
\end{equation}
\end{proposition}

Indeed, taking $\delta_1 = A_\varepsilon N^{-\frac{1}{5c}+\varepsilon}$ and $\delta_2 =  B_\varepsilon N^{-(1-\frac{1}{c}+\varepsilon)}$ for constants $A_\varepsilon, B_\varepsilon>0$, we see that \[ \frac{\delta_1}{\delta_1+\delta_2} \ll_\varepsilon N^{(1-\frac{1}{c}+\varepsilon)-(\frac{1}{5c}-\varepsilon)} = N^{-(\frac{6}{5c}-1)+2\varepsilon}. \]

The rest of the paper will be concerned with establishing (\ref{eqn:negative}) and (\ref{eqn:large}). In both cases the presence of the indicator function $1_{S_h}$ in its present form is difficult to work with. However, the following lemma, which can be found in \cite{mirek2015roth}, for example, allows us to write it in a much more convenient way.

\begin{lemma} \label{lem:rewriteindicator}
We have \[ 1_{S_h}(n) = \lfloor -\eta(n) \rfloor - \lfloor -\eta(n+1) \rfloor \] for all $n$ sufficiently large (depending on $h$).
\end{lemma}

\begin{proof}
For all $n$ sufficiently large depending on $h$ we have \[ \eta(n+1)-\eta(n) \leq \eta'(n)=\frac{1}{h'(\eta(n))} \ll_\varepsilon n^{\frac{1}{c}-1+\varepsilon} \] by the Mean Value theorem and the fact that $h'$ and $\eta$ have growth rates equal to $c-1$ and $\frac{1}{c}$ respectively. Thus for all $n$ sufficiently large we see that $|\eta(n+1)-\eta(n)|<\frac{1}{2}$. Thus \[ \lfloor - \eta(n) \rfloor - \lfloor - \eta(n+1) \rfloor \in \{0,1\}. \] It now suffices to show that $n \in S_h$ if and only if $\lfloor - \eta(n) \rfloor - \lfloor -\eta(n+1) \rfloor = 1$.

First suppose $n \in S_h$. Thus there exists $m \geq N_h$ such that \[ n=\lfloor h(m) \rfloor, \] or equivalently \[ \eta(n) \leq m < \eta(n+1). \] Then \[ \eta(n+1) \leq \eta(h(m)+1) \leq \eta(h(m+1)) = m+1 \] since \begin{equation} \label{eqn:hardyplusone} h(m+1)\geq h(m)+h'(m) \geq h(m)+1, \end{equation} where the second inequality follows by the Mean Value theorem and the fact that $h'(t) > 1$ for all $t$ sufficiently large. It now follows that $\lfloor -\eta(n+1) \rfloor = -(m+1)$ and $\lfloor -\eta(n) \rfloor \geq -m$. On the one hand, this means \[ \lfloor - \eta(n) \rfloor - \lfloor -\eta(n+1) \rfloor \geq -m+(m+1)=1, \] and on the other hand we see that \[ \lfloor - \eta(n) \rfloor - \lfloor -\eta(n+1) \rfloor \leq \eta(n+1)-\eta(n)+1 \leq \eta'(n)+1 = \frac{1}{h'(\eta(n))}+1<2. \] Thus \[ \lfloor - \eta(n) \rfloor - \lfloor -\eta(n+1) \rfloor=1. \]

Now suppose $\lfloor - \eta(n) \rfloor - \lfloor -\eta(n+1) \rfloor=1$. Then \[ \lfloor -\eta(n) \rfloor = \lfloor -\eta(n+1) \rfloor+1 \leq 1-\eta(n+1) \leq -\eta (n) \] by (\ref{eqn:hardyplusone}) again, and so \[ \eta(n) \leq -1 - \lfloor - \eta(n+1) \rfloor <-1 +\eta(n+1)-1 = \eta(n+1). \] On taking $m = -1-\lfloor - \eta(n+1) \rfloor$, we see that \[ \eta(n) \leq m<\eta(n+1), \] or equivalently \[ n \leq h(m)<n+1. \] The result follows.
\end{proof}

We will use Lemma \ref{lem:rewriteindicator} to prove both (\ref{eqn:negative}) and (\ref{eqn:large}). We begin with (\ref{eqn:large}) since this will be easier.

\begin{proof}[Proof of (\ref{eqn:large})]
Writing out the definition of $\Psi_N$ and applying Lemma \ref{lem:rewriteindicator}, we have \[ \E_{n \in [N]} \Psi_N(n) = \E_{n \in [N]} \left(1-\frac{n}{N} \right) (\lfloor -\eta(n) \rfloor - \lfloor - \eta(n+1) \rfloor) + O(N^{-1}). \] By utilising non-negativitity and bounding \[ 1-\frac{n}{N} \geq \frac{1}{2} 1_{n \leq \frac{N}{2}}, \] we see that \begin{multline*} \E_{n \in [N]} \left(1-\frac{n}{N} \right) (\lfloor - \eta(n) \rfloor - \lfloor \eta(n+1) \rfloor) \gg \E_{n \in [\frac{N}{2}]} (\lfloor - \eta(n) \rfloor - \lfloor \eta(n+1) \rfloor) \\ = \frac{\lfloor - \eta(1) \rfloor - \lfloor - \eta(\lfloor \frac{N}{2} \rfloor + 1) \rfloor}{N} = \frac{\eta(\lfloor \frac{N}{2} \rfloor+1)+O(1)}{N} \gg_\varepsilon N^{-(1-\frac{1}{c}+\varepsilon)} + O(N^{-1}), \end{multline*} where we used the fact that $\eta$ has growth rate $\frac{1}{c}$ for the last inequality. Combining these two bounds we obtain \[ \E_{n \in [N]} \Psi_N(n) \gg_\varepsilon N^{-(1-\frac{1}{c}+\varepsilon)} + O(N^{-1}) \gg N^{-(1-\frac{1}{c}+\varepsilon)} \] as required.

\end{proof}

The proof of (\ref{eqn:negative}) will be more involved. We were able to exploit telescoping behaviour in the proof of (\ref{eqn:large}) which did not need any manipulation of the expression $\lfloor - \eta(n) \rfloor - \lfloor - \eta(n+1) \rfloor$. However, we will not be able to exploit this structure to prove (\ref{eqn:negative}). Instead, we will use a technique first used in \cite{heath1983pjateckiui}.

Define the sawtooth function by \[ \phi(\xi) = \{\xi\}-\frac{1}{2} = \xi - \lfloor \xi \rfloor - \frac{1}{2} \] for all $\xi \in \R$. Therefore, we can write \begin{multline} \label{eqn:insertsawtooth}
    \E_{n \in [N]} \Psi_N(n) \cos(2 \pi \xi n) = \E_{n \in [N]} F_N^\Z(n) \cos(2\pi \xi n) (\eta(n+1)-\eta(n)) \\ + \E_{n \in [N]} F_N^\Z(n) \cos(2\pi \xi n) (\phi(\eta(n+1)) - \phi(\eta(n))). 
\end{multline} Then by expanding $\phi$ into its Fourier series, one has \begin{equation}
    \phi(\xi) = \sum_{0<|m| \leq M} \frac{1}{2\pi im} e(m\xi) + O \left( \min \left(1,\frac{1}{M \|\xi\|_\T} \right) \right)
\end{equation} for any $M \in \N$. Substituting this into (\ref{eqn:insertsawtooth}), we see that \begin{equation} \label{eqn:errorterms}
    \E_{n \in [N]} \Psi_N(n) \cos(2\pi \xi n) = \E_{n \in [N]} F_N^\Z(n) \cos(2\pi \xi n) (\eta(n+1)-\eta(n)) + I_1+I_2+I_3
\end{equation} where \begin{equation}
    I_1 = \sum_{0<|m| \leq M} \frac{1}{2\pi i m} \E_{n \in [N]} F_N^\Z(n) \cos(2\pi \xi n) (e(m\eta(n+1))-e(m\eta(n))),
\end{equation} \[
    I_2 \ll \E_{n \in [N]} \min \left(1,\frac{1}{M \| \eta(n+1) \|_\T} \right), \] and \begin{equation} \label{eqn:I3}
    I_3 \ll \E_{n \in [N]} \min \left(1,\frac{1}{M \|\eta(n)\|_\T} \right).
\end{equation}

Our strategy will now be as follows. We will bound \begin{equation} \label{eqn:I2bound} I_2,I_3 \ll_\varepsilon M^{-1} \log M + \log M M^\frac{1}{2} N^{-\frac{1}{2c}+\varepsilon} \end{equation} and \begin{equation} \label{eqn:I1bound} I_1 \ll_\varepsilon M^\frac{3}{2} N^{\frac{1}{2c}-1+\varepsilon}. \end{equation} Since $c>1$ we may bound $N^{\frac{1}{2c}-1} \leq N^{-\frac{1}{2c}}$ so \[ I_1+I_2+I_3 \ll_\varepsilon M^{-1} \log M + M^\frac{3}{2} N^{-\frac{1}{2c}}. \] Minimising over $M \in \N$ we see that one should take $M$ on the order of $N^\frac{1}{5c}$. This will yield a bound of the form \[ I_1+I_2+I_3 \ll_\varepsilon N^{-\frac{1}{5c}+\varepsilon}. \] We will then show that the main term in (\ref{eqn:errorterms}) does not get too negative. The intuition here is if we consider $\eta(n+1)-\eta(n)$ as a harmless non-negative weight, we can factor it out. Since it is non-negative, we then need to bound \[ \E_{n \in [N]} F_N^\Z(n) \cos(2\pi \xi n) \] below. Since $F_N^\T$ is non-negative, we expect this to be bounded below by a constant times $-N^{-1}$, and this is indeed the content of Lemma \ref{lem:nonnegative}. Combining all bounds we will have shown \[ \E_{n \in [N]} \Psi_N(N) \cos(2\pi \xi n) \gg_\varepsilon -N^{-1} + O_\varepsilon(N^{-\frac{1}{5c}+\varepsilon}) \gg_\varepsilon -N^{-\frac{1}{5c}+\varepsilon}, \] which will establish (\ref{eqn:negative}). The proof of Lemma \ref{lem:findfunction}, and hence of Theorem \ref{thm:mainthm}, will thus be complete.

\section{Bounding the Error Terms}

\subsection{Bounds on $I_2$ and $I_3$}

In this subsection we establish (\ref{eqn:I2bound}). We will just prove it for $I_3$ as the argument for $I_2$ is analogous.

One can expand \begin{equation} \label{eqn:expandmin} \min \left( 1 ,\frac{1}{M \|\xi\|_\T} \right) = \sum_{m \in \Z} b_m e(m\xi) \end{equation} where \[ b_m \ll \min \left( \frac{\log M}{M}, \frac{M}{m^2} \right). \] for all $m \in \Z$. Substituting (\ref{eqn:expandmin}) into (\ref{eqn:I3}), we see that \begin{equation} \label{eqn:splitI3} I_3 \ll \frac{\log M}{M} + \sum_{0<|m| \leq M} \frac{\log M}{M} |\E_{n \in [N]} e(m \eta(n))| + \sum_{|m|>M} \frac{M}{m^2} |\E_{n \in [N]} e(m\eta(n))|. \end{equation} Fix $m \neq 0$ and write \[ \E_{n \in [N]} e(m\eta(n)) = \frac{1}{N} \sum_{j \leq \log_2 N} \sum_{\substack{2^j \leq n<2^{j+1} \\ n \leq N}} e(m\eta(n)). \] Observe that $\eta''$ has growth rate $\frac{1}{c}-2$, which is negative hence $|\eta''|$ is decreasing. Then by Lemma \ref{lem:vdc}, Van der Corput's second derivative test, we obtain \[ \E_{n \in [N]} e(m\eta(n)) \ll \frac{1}{N} \sum_{j \leq \log_2N} 2^j |m|^\frac{1}{2} |\eta''(2^j)| |\eta''(2^{j+1})|^{-\frac{1}{2}} + |m|^{-\frac{1}{2}} |\eta''(2^{j+1})|^{-\frac{1}{2}} \] Observe that the function \[ t \mapsto t \eta''(t) \] has growth rate $\frac{1}{c}-1$. In particular, it has a negative growth rate hence it is decreasing in absolute value. It follows that \[ 2^j |\eta''(2^j)| |\eta''(2^{j+1})|^{-\frac{1}{2}} \ll |\eta''(2^{j+1})|^{-\frac{1}{2}}. \] We can therefore bound \begin{equation} \label{eqn:exponentialsumbound} \E_{n \in [N]} e(m\eta(n))  \ll \frac{1}{N} \sum_{j \leq \log_2N} |m|^\frac{1}{2} |\eta''(2^{j+1})|^{-\frac{1}{2}} \ll_\varepsilon |m|^\frac{1}{2} N^{-\frac{1}{2c}+\varepsilon}. \end{equation} Putting this bound into (\ref{eqn:splitI3}) yields (\ref{eqn:I2bound}).

\subsection{Bounds on $I_1$}

We will now prove (\ref{eqn:I1bound}). To begin, we rewrite \[ I_1 = I_1' + I_1'' \] where \[ I_1' = \sum_{0<|m| \leq M} \frac{1}{4 \pi i m} \E_{n \in [N]} F_N^\Z(n) e(m\eta(n)+\xi n) (e(m\eta(n+1)-m\eta(n))-1) \] and \[ I_1'' = \sum_{0<|m| \leq M} \frac{1}{4 \pi i m} \E_{n \in [N]} F_N^\Z(n) e(m\eta(n)-\xi n) (e(m\eta(n+1)-m\eta(n))-1). \] It now suffices to prove \[ I_1', I_1'' \ll_\varepsilon M^\frac{3}{2} N^{-\frac{1}{2c}+\varepsilon}. \] We will prove this for $I_1'$ since $I_1''$ can be handled analogously. Write \[ \psi_m(n) = F_N^\Z(n) (e(m\eta(n+1)-m\eta(n))-1) \] and apply summation by parts to obtain \[ I_1' = \sum_{0<|m| \leq M} \frac{1}{4\pi i m} \left( \psi_m(N) \E_{n \in [N]} e(m\eta(n)+\xi n) - \sum_{N_0<N} \Delta \psi_m(N_0) \frac{1}{N} \sum_{n \in [N_0]} e(m \eta(n)+\xi n) \right). \] Observe the bounds \[ |\psi_m(n)| \ll |m||\eta'(n)| \] and \[ |\psi_m(n+1)-\psi_m(n)| \ll |m| |\eta''(n)| + N^{-1} |m| |\eta'(n)|. \]  Since the functions $m \eta(t)+\xi t$ and $m \eta(t)$ have the same second derivative, we see by the same argument used to establish (\ref{eqn:exponentialsumbound}) that \[ \frac{1}{N} \sum_{n \in [N_0]} e(m\eta(n)+\xi n) \ll_\varepsilon |m|^\frac{1}{2} N^{-1} N_0^{1-\frac{1}{2c}+\varepsilon}. \] Thus \[ I_1' \ll_\varepsilon  \sum_{0<|m| \leq M} \frac{1}{|m|} \left( |m|^\frac{3}{2} N^{\frac{1}{2c}-1+\varepsilon} + \sum_{N_0<N} |m|^\frac{3}{2} N_0^{\frac{1}{2c}-1+\varepsilon} N^{-1} \right) \ll M^\frac{3}{2} N^{\frac{1}{2c}-1+\varepsilon}, \] and so (\ref{eqn:I1bound}) follows.

\section{Non-Negativity of the Main Term}

Finally, we show that \begin{equation} \label{eqn:nonnegative}
    \E_{n \in [N]} F_N^\Z(n) \cos(2\pi \xi n)(\eta(n+1)-\eta(n)) \gg - N^{-1}.
\end{equation}

To begin, let $\psi(n) = \eta(n+1)-\eta(n)$ and apply summation by parts: \[ \E_{n \in [N]} F_N^\Z(n) \cos(2 \pi \xi n) = \psi(N) \E_{n \in [N]} F_N^\Z(n) \cos(2 \pi \xi n) - \sum_{N_0<N} \Delta \psi(N_0) \frac{1}{N} \sum_{n \in [N_0]} F_N^\Z(n) \cos(2 \pi \xi n). \] Note that \[ \E_{n \in [N]} F_N^\Z(n) \cos(2\pi \xi n) = -\frac{1}{N} + \frac{1}{2N} \sum_{n=-N}^N F_N^\Z(n) e(n \xi) = -\frac{1}{N} + \frac{1}{2N} F_N^\T(\xi), \] so by Lemma \ref{lem:nonnegative} we see that \[ \psi(N) \E_{n \in [N]} F_N^\Z(n) \cos(2\pi \xi n) \gg_\varepsilon -N^{\frac{1}{c}-2+\varepsilon}. \]

Now note that \[ \sum_{n \in [N_0]} F_N^\Z(n) \cos(2\pi \xi n) = -1+\frac{1}{2} \sum_{n=-N_0}^{N_0} F_N^\Z(n) e(n\xi). \] Since \[ F_N^\Z(n) = \int_\T F_N^\T(\zeta) e(-n\zeta) \,d\zeta,  \] we see that \[ \sum_{n=-N_0}^{N_0} F_N^\Z(n) e(n\xi) = \int_\T F_N^\T(\zeta) \sum_{n=-N_0}^{N_0} e(n(\xi-\zeta)) \,d\zeta = (F_N^\T * D_{N_0})(\xi). \] By Lemma \ref{lem:nonnegative}, we can write \[ F_N^\T * D_{N_0} = \frac{1}{N} \sum_{n=0}^{N-1} D_n * D_{N_0}. \] It can be calculated directly that \[ D_n * D_{N_0} = D_{\min(n,N_0)}, \] or alternatively this can be seen by taking Fourier transforms of the identity \[ 1_{[-n,n]} 1_{[-N_0,N_0]} = 1_{[-\min(n,N_0),\min(n,N_0)]}. \] Therefore, \[ F_N^\T * D_{N_0} = \frac{1}{N} \sum_{n=0}^{N-1} D_n * D_{N_0} = \frac{N_0}{N} F_{N_0}^\T + \frac{N-N_0}{N} D_{N_0} \] and we can expand \begin{multline*}
    -\sum_{N_0<N} \Delta \psi(N_0) \frac{1}{N} \sum_{n \in [N_0]} F_N^\Z(n) \cos(2\pi\xi n) \\ = \frac{1}{N} \sum_{N_0<N} \Delta \psi(N_0) - \frac{1}{2N} \sum_{N_0<N} \Delta \psi(N_0) \frac{N_0}{N} F_{N_0}^\T(\xi)  - \frac{1}{2N} \sum_{N_0<N} \Delta \psi(N_0) \frac{N-N_0}{N} D_{N_0}(\xi).
\end{multline*} We can immediately bound \[ \frac{1}{N} \sum_{N_0<N} \Delta \psi(N_0) \gg -N^{-1} \] and \[ -\frac{1}{2N} \sum_{N_0<N} \Delta \psi(N_0) \frac{N_0}{N} F_{N_0}^\T(\xi) \geq 0 \] using Lemma \ref{lem:nonnegative}. Finally, write $\frac{N-N_0}{N} = F_N^\Z(N_0)$ and apply summation by parts to obtain \begin{multline*}
    \sum_{N_0<N} \Delta \psi(N_0) F_N^\Z(N_0) D_{N_0}(\xi) \\ = \Delta \psi(N-1) F_N^\Z(N-1) \sum_{N_0<N} D_{N_0}(\xi) - \sum_{N_0<N-1} \Delta(\Delta \psi(N_0) F_N^\Z(N_0)) \sum_{n \leq N_0} D_N(\xi),
\end{multline*} which we can simplify slightly to \[ \Delta \psi(N-1) F_N^\Z(N-1) NF_N^\T(\xi)  - \sum_{N_0<N-1} \Delta(\Delta \psi(N_0) F_N^\Z(N_0)) (N_0+1) F_{N_0+1}^\T(\xi) \] by Lemma \ref{lem:nonnegative}. By Lemma \ref{lem:nonnegative} again we can bound \[ \Delta \psi(N-1) F_N^\Z(N-1) N F_N^\T(\xi) \ll 1 \] and we can calculate \[ \Delta(\Delta \psi(N_0) F_N^\Z(N_0)) = F_N^\Z(N_0)\Delta^2 \psi(N_0) - \frac{1}{N} \Delta \psi(N_0+1) \gg -1 \] so \[ -\sum_{N_0<N-1} \Delta(\Delta \psi(N_0) F_N^\Z(N_0)) (N_0+1) F_{N_0+1}^\T(\xi) \ll 1 \] as well. Therefore, \[ -\frac{1}{2N} \sum_{N_0<N} \Delta \psi(N_0) \frac{N-N_0}{N_0} D_{N_0}(\xi) \gg -N^{-1}. \]

Putting everything together, we obtain (\ref{eqn:nonnegative}) and thus the proof of Theorem \ref{thm:mainthm} is complete.

\printbibliography
\end{document}